\documentclass[letterpaper, 10 pt, conference]{ieeeconf}

\IEEEoverridecommandlockouts    
\usepackage{amsthm}

\usepackage{amsmath}
\usepackage{amssymb,amsfonts,amscd}
\usepackage{pifont} 
\usepackage{bm} 

\usepackage[table]{xcolor}
\usepackage{graphicx}
\usepackage{float}   
\usepackage{enumerate}
\usepackage{hyperref}
\usepackage{booktabs} 
\usepackage{array}   

\newtheorem{theorem}{Theorem}
\newtheorem{definition}[theorem]{Definition}

\newtheorem{lemma}[theorem]{Lemma}

\newtheorem{proposition}[theorem]{Proposition}
\newtheorem{remark}[theorem]{Remark}

\usepackage{lipsum}

\definecolor{verylightgray}{gray}{0.95}
\newcolumntype{L}{>{$}l<{$}} 
\newcolumntype{G}{>{\columncolor{verylightgray}$}l<{$}}

\def\R{\mathbb{R}}

\def\phi{\varphi}

\def\xx{\mathbf{x}}
\def\uu{\mathbf{u}}
\def\ww{\mathbf{w}}
\def\zz{\mathbf{z}}

\def\Sin{s_{\text{in}}}
\def\Vin{v_{\text{in}}}
\def\Win{w_{\text{in}}}
\def\Wout{w_{\text{out}}}
\def\Pin{P_{\text{in}}}
\def\Pout{P_{\text{out}}}
\def\Py{P_{\text{yield}}}

\def\mumax{\mu_{\max}}
\def\rhomax{\rho_{\max}}
\def\phimax{\varphi_{\max}}

\def\qmin{q_{\min}}

\def\partials#1#2{\frac{\partial#1}{\partial#2}}

\def\pp#1{\left(#1\right)}
\def\bb#1{\left[#1\right]}
\def\cb#1{\left\{#1\right\}}

\def\transp#1{{#1}^{\top}}

\def\AA{\mathcal{A}}
\def\DD{\mathcal{D}}

\def\UU{\mathcal{U}}
\def\WW{\mathcal{W}}

\def\BB{\mathcal{B}}
\def\ZZ{\mathcal{Z}}

\DeclareMathOperator{\epi}{epi}
\DeclareMathOperator{\dom}{dom}

\title{\LARGE \bf
Single- and multi-objective performance optimization of an algal-bacterial
synthetic process
}

\author{Rand Asswad$^{1,2}$, Jean-Luc Gouzé$^{3,\dag}$, Eugenio Cinquemani$^{1,2,\dag}$
\thanks{$^\dag$These authors contributed equally to the work}%
\thanks{$^1$Université Grenoble Alpes, Inria, 38000 Grenoble, France\newline%
    {\tt rand.asswad@inria.fr} ; %
    {\tt eugenio.cinquemani@inria.fr}}%
\thanks{$^2$Université Grenoble Alpes, CNRS, LIPhy, 38000 Grenoble, France}
\thanks{$^3$Université Côte d’Azur, Inria, INRAE, CNRS, MACBES Team, 06902 Sophia Antipolis, France
    {\tt jean-luc.gouze@inria.fr}}%
}

\begin{document}

\maketitle
\thispagestyle{empty}
\pagestyle{empty}

\begin{abstract}
Microalgae are an important source of precursors (e.g. lipids) for a variety of biosynthetic processes (e.g. biofuel production). Their co-culturing with other organisms providing essential substrates for growth may reduce cost and provide new handles to control and robustify the production process. In previous work, we have introduced a nonlinear ordinary differential equation model for an optogenetically controllable algal-bacterial consortium, and studied maximization of algal biomass productivity in a continuous-flow bioreactor relative to optogenetic action and dilution rate.      

In this work, we expand the investigation of steady-state production performance for different objective criteria and control knobs. We additionally consider a yield criterion and a cost criterion, as well as a multiobjective optimization problem whose solution is shown to directly relate with a notion of net process profit. We investigate dependence of the optimal solutions on all the available bioprocess control knobs (optogenetics, dilution rate, richness of input medium), providing analytical results to characterize the solutions from different criteria and the relations among them, as well as simulations illustrating our results for a realistic set of biological system parameters.
\end{abstract}

\section{Introduction}\label{sec:intro}

Microalgae play a crucial role in biotechnology, offering diverse applications
across industrial and environmental sectors, including animal feed, fertilizers,
and pharmaceutical production \cite{rizwan_2018}.
Notably, their high lipid content makes them a promising source for biofuel
production. Amid the ongoing global energy crisis, the demand for renewable
energy sources has intensified, driving increased interest in biofuels as
a sustainable alternative to conventional fossil fuels
\cite{polat_2022,sajjadi_2018}.
The highlighted potential of algal biofuel raises the need to optimize
biofuel production processes. Biotechnological and economical criteria
are to be considered to render algae fuel a viable option
\cite{nodooshan_2018,jayaraman_2015}.
While mono-cultures remain dominant in industrial biotechnology,
synthetic microbial consortia have several advantages over single species
such as increased performance and resilience, compartmentalization, and
modular functionality. In particular, synthetic or catalytic properties
of present strains can be employed to replace external nutrient or catalyst
input feeds \cite{rapp_2020,shong_2012}.
Despite the fact that coexistence of different microbial strains is
ubiquitous in nature, maintaining coexistence in synthetic co-cultures
presents a major challenge due to the added complexity to the system
dynamics \cite{shong_2012,martinez_2023}.

We have presented in a previous work \cite{asswad_cdc2024}
a synthetic algal-bacterial consortium model describing the growth of
\textit{Auxenochlorella protothecoides} microalgae along with a strain of
\textit{Escherichia coli} bacteria in a continuous-flow stirred-tank bioreactor.
Bacteria are modified in order to synthesize \textit{thiamine} vitamin along its growth,
which is a limiting substrate to the growth of the algae.
Vitamin secretion is controlled via optogenetics \cite{raghavan_2020},
allowing modulating the proportion of bacterial metabolites dedicated to vitamin
synthesis and that for bacterial growth \cite{jong_2017}.
We have extensively studied in \cite{asswad_cdc2024} the system's
steady states and provided mathematical characterization for existence
and stability conditions.
We considered the maximization of a single criterion: microalgal productivity
(\textit{i.e.} harvested algal biomass from the bioreactor)
via the dilution rate and the optogenetic control. Two optimal control
problems were investigated: a static optimal control problem where the productivity
is maximized at the coexistence steady state using constant control inputs,
and a dynamic optimal control problem where productivity is maximized over
a finite time horizon using time-varying control functions.


Other authors have been dedicating attention to model-based analysis, optimization and even feedback control of microbial consortia. To give some examples, in \cite{harvey_2014}, an in-depth mathematical analysis of division of labor in a microbial consortium is provided, with the goal of explaining observed enhanced productivity of consortia relative to single species.
In \cite{bayen_2019}, optimization of biogas production at steady-state in a two-stage anaerobic digestion model in explored.
In \cite{treloar_2020,salzano_2022}, feedback control strategies are considered with the ultimate goal to optimize cooperative bioproduction of prototypical microbial consortia. A discussion of recent advances, challenges and perspectives of control of microbial populations from a modelling and optimization viewpoint is provided in 
\cite{bertaux_2021}.
Despite the higher productivity that can be achieved with dynamic control \cite{asswad_cdc2024}, 
optimization of productivity or other objective criteria at steady state is interesting from a practical viewpoint
as constant controls are easy to implement in biotechnological applications.
The relative simplicity of the resulting problems compared to dynamic counterparts simplifies exploration of problems in several control variables, and it allows for the study of multi-objective optimization problems. It is also worth remarking that solution of a steady-state problem is often the starting point for the implementation of a robust control strategy around the optimal state sought~\cite{treloar_2020,salzano_2022}. 

In this article, we extend the steady-state optimization analysis of our previous work~\cite{asswad_cdc2024} in several directions. We introduce several single and multiple optimization criteria, and study the optimal solutions relative to all optimization parameters entering the model for the microbial consortium biosynthesis process (namely dilution rate, optogenetic control and richness of the bacterial growth substrate provided as input). In more detail, we first review the consortium model and the conditions for
algal-bacterial coexistence in Section \ref{sec:model}.
In Section \ref{sec:pb}, we introduce the steady-state optimization problems and
formalize the relevant decision variables. 
We review the criterion of productivity explored in~\cite{asswad_cdc2024}, 
and further consider bioreactor yield  as well as a multi-objective optimization
framework that is shown to be directly related with a concept of bioreactor's
running net profit.
In Section \ref{sec:opti}, we analyze these optimization problems with respect
to two decision variables (dilution rate and optogenetic control),
establishing properties of
existence and uniqueness of solutions (scalar criteria) and of the Pareto solutions (multiple criteria). We support our findings with numerical
simulations based on reference parameters.
Section \ref{sec:feed} extends this analysis by considering the remaining 
decision variable (substrate richness).
We investigate its impact on the mathematical properties established
in Section \ref{sec:opti} and its influence on the solutions.
Simulations and mathematical analysis are provided to guide the appropriate
selection of this added variable.
Finally, we summarize our findings and discuss future research perspectives in
Section \ref{sec:conclusion}.

\section{Algal-bacterial consortium model}\label{sec:model}

The model we consider is the consortium presented in
\cite{asswad_cdc2024}, describing a co-culture of
\textit{Escherichia coli} bacteria and of a microalgal strain of the family of chlorella  
(\textit{Auxenochlorella protothecoides}) in a continuously stirred-tank bioreactor.
This algal strain has the peculiarity to need vitamin B1 (thiamine) for growth.
In this consortium, the vitamin is synthesized by the (suitably engineered)
bacteria in dependence of the action of optogenetic control.
\par
We denote by $s$ [$g\cdot L^{-1}$], $e$ [$g\cdot L^{-1}$],
$v$ [$mg\cdot L^{-1}$], and $c$ [$g\cdot L^{-1}$]
the glucose substrate, \textit{E. coli} biomass, secreted vitamin,
and chlorella biomass concentration in the (fixed volume) bioreactor,
in the same order. We further denote by $q$ [$mg\cdot g^{-1}$] the
internal algal quota of the vitamin.
The state of the system at time $t$ is defined as
$\xx(t)=\transp{(s(t),e(t),v(t),q(t),c(t))}$.
The system dynamics are given by a cascade of a Monod model for \textit{E. coli} growth on glucose,
modified to account for the synthesis of vitamins under optogenetic control,
and of a so-called variable yield Droop model (see~\cite{asswad_cdc2024} and references therein)
for algal growth as a function of vitamin availability. They take the form
\begin{align}
\dot{s} &= -\frac{1}{\gamma}\varphi(s)e + d(\Sin - s)\label{eq:ds}\\
\dot{e} &= (1 - \alpha)\varphi(s)e - de\label{eq:db}\\
\dot{v} &= \alpha\beta\varphi(s)e - \rho(v)c - dv\label{eq:dv}\\
\dot{q} &= \rho(v) - \mu(q)q\label{eq:dq}\\
\dot{c} &= \mu(q)c - dc\label{eq:da}
\end{align}
where the functions $\phi$ and $\rho$, defined over $[0,\infty)$,
and $\mu$, defined over $[\qmin,\infty)$, have the expressions
\begin{equation*}
\phi(s)=\frac{\phimax s}{k_s+s},
\rho(v)=\frac{\rhomax v}{k_v+v},
\mu(q)=\mumax\pp{1-\frac{\qmin}{q}}.
\end{equation*}
Let $\Omega=\R_+^5\setminus\cb{q<\qmin}$ be the state space.
In fact, $\xx(0)\in\Omega$ implies that $\xx(t)\in\Omega$ for all $t\geq 0$.

The variables $d$ and $\Sin$ are operational parameters of the
bioreactor representing the dilution rate and the input substrate
feed respectively. As for $\alpha\in[0,1]$ appearing in
\eqref{eq:db}-\eqref{eq:dv} represent the proportion of glucose
resources going towards vitamin synthesis while the remaining
$1-\alpha$ goes toward bacterial growth.
The factor $\gamma$ represents the bacterial growth yield and
$\beta$ represents the vitamin synthesis yield. The functions
$\phi(s)$ and $\mu(q)$ are the growth rates per capita for
bacteria and algae respectively, and $\rho(v)$ is the vitamin
uptake rate per capita of algae.
In fact, the yields constants $\gamma$ and $\beta$ as well as
the parameters defining $\phi,\rho$, and $\mu$ are intrinsic
biological parameters that are reported in Table \ref{tab:jc-params}
taken from \cite{asswad_cdc2024}.

\begin{table}[t]
    \centering
    \caption{Biological model parameters from \cite{asswad_cdc2024}}
    \label{tab:jc-params}
    \begin{tabular}{|G|LL|G|LL|}
        \hline
        k_v & 0.57 & mg{\cdot}L^{-1} &
        k_s & 0.09& g{\cdot}L^{-1} \\
        \rhomax & 27.3 & mg{\cdot}g^{-1}{\cdot}\text{day}^{-1}&
        \phimax & 6.48 & \text{day}^{-1}\\
         q_{\min} & 2.76 &  mg{\cdot}g^{-1} &
        \gamma & 0.44 & g{\cdot}g^{-1}\\
        \mumax & 1.02 & \text{day}^{-1} &
        \beta & 23 & mg{\cdot}g^{-1}\\
        \hline
    \end{tabular}
\end{table}

\subsection{Coexistence at steady states}

\begin{table}[b]
    \centering
    \caption{Existence of equilibria over $\Omega$
    and their stability over
    $\Omega\setminus\cb{e=0\cup c=0}$
    with $d_1(\alpha,\Sin)=\psi_\alpha(\Sin)$
    and $d_2(\alpha,\Sin)=(1-\alpha)\phi(\Sin)$}
    \label{tab:eq_stability}
    \begin{tabular}{|G|c|c|c|}
        \hline\rowcolor{verylightgray}
        & $0   < d < d_1$
        & $d_1 < d < d_2$
        & $d_2 < d$\\\hline
        \xx_0     & unstable & unstable & GAS\\
        \xx_{1,0}    & unstable & GAS & --\\
        \xx_{1,1}    & GAS      & -- & --\\\hline
    \end{tabular}
\end{table}

In our previous work \cite{asswad_cdc2024}, we have
proved the existence of a washout equilibrium at
$\xx_0=\transp{(\Sin,0,0,\qmin,0)}$ regardless of the system
parameters. Furthermore, an algal washout equilibrium exists
at $\xx_{1,0}=\transp{(s^*,e^*,\Vin^*,q_0,0)}$ if
$d<(1-\alpha)\phi(\Sin)$ with
\begin{equation}\label{eq:bacterial-eq}
\begin{aligned}
s^*=\phi^{-1}\pp{\frac{d}{1-\alpha}},\quad
&e^*=(1-\alpha)\gamma(\Sin - s^*),\\
\Vin^*=\alpha\beta\gamma(\Sin-s^*),\quad
&q_0=\qmin+\rho(\Vin^*)/\mumax.
\end{aligned}
\end{equation}
Finally, a coexistence equilibrium state may exist
at $\xx_{1,1}=\transp{(s^*,e^*,v^*,q^*,c^*)}$ with
\begin{equation}\label{eq:algal-eq}
v^*=\rho^{-1}(d\cdot\mu^{-1}(d)),
~q^*=\mu^{-1}(d),~c^*=\frac{\Vin^*-v^*}{q^*}.
\end{equation}
The coexistence equilibrium exists if $d<\psi_\alpha(\Sin)$
where $\psi_\alpha$ is an increasing function defined such that 
\begin{equation}\label{eq:psi-alpha}
\psi_\alpha^{-1}(y) = \phi^{-1}\pp{\frac{y}{1-\alpha}}
    + \frac{\rho^{-1}(y\cdot\mu^{-1}(y))}{\alpha\beta\gamma}.
\end{equation}
Stability and existence of the equilibria, as reported in Table
\ref{tab:eq_stability}, are studied in \cite{asswad_cdc2024}.

\section{Single- and multi-objective optimization problems}\label{sec:pb}

In this section we introduce the optimization problems of interest.
Having established the necessary and sufficient conditions for
the existence of the coexistence equilibrium $\xx_{1,1}$, referred to
from this point onward as $\xx^*$,
we have yet to introduce the decision variables that constitute the
control $\uu$ from the appropriate set of admissible controls $\UU$
that ensure the coexistence at the functional steady state $\xx^*(\uu)$.
Then we consider the optimization criteria to take into account
that are relevant in biotechnology.

\subsection{Decision variables}

Existence of equilibria as well as their values
are characterized by the system's parameters as reported
in the previous section (see Table \ref{tab:eq_stability}).
While the parameters reported in Table \ref{tab:jc-params}
are intrinsic to the system, relating to the bacterial and algal
biological properties, the dilution rate $d$ and the substrate feed $\Sin$
are operational parameters of the bioreactor that can be adjusted as best
for the biotechnological process.
Finally, the $\alpha\in(0,1)$ models the amount of resources that
are allocated for bacterial growth or for vitamin synthesis.
In this sequel, we consider optogenetic control through $\alpha$
following \cite{asswad_cdc2024,raghavan_2020}.

In short, we consider the parameters $\alpha,d,$ and $\Sin$
as decision variables for the control problem. We show later
that performance improves as $\Sin$ increases
\cite{martinez_2023, mauri_2020},
we therefore consider in section \ref{sec:opti} the problem
for a fixed $\Sin$ with free control variables $\uu=(\alpha,d)$
to be chosen in $\UU$, such as $\UU$ is the set of all values of
$(\alpha,d)$ for which the coexistence steady state $\xx^*(\uu)$
is well-defined.
We expand the scope in section \ref{sec:feed} to study
the effect of $\Sin$ on the problem properties and its solution,
where $\ww=(\alpha,d,\Sin)\in\WW$ for the appropriate $\WW$.

\subsection{Optimization criteria}

Naturally, we aim to maximize the amount of microalgae
that is harvested from the bioreactor which is the quantity
\begin{equation}\label{eq:Pout}
\Pout = d\cdot c^* =
    \frac{\alpha\beta\gamma d\pp{\Sin - \psi_\alpha^{-1}(d)}}{\mu^{-1}(d)},
\end{equation}
and is commonly referred to
as the \textbf{microalgal productivity}.
Another quantity to consider is the glucose feed
that enters the bioreactor, characterized as $\Pin=d\cdot\Sin$.
The \textbf{nutrient feed} being a cost to the bioprocess,
we minimize this amount.
A common metric for the efficiency of a bioreactor synthesis
process is the \textbf{bioreactor yield}, which is the ratio
of the harvested target biomass with respect to the nutrient feed
\cite{martinez_2023}, here simply
$\Py=\Pout/\Pin=c^*/\Sin$.

While the bioreactor yield combines the input and the target
output of the bioreactor, it only takes into account the
volumetric efficiency of the process. However, this criterion
does not account for the running economical cost and
the profitability of the bioprocess. From a biotechnological
standpoint, the bioreactor efficiency is optimized by
maximizing the microalgal productivity $\Pout$ and minimizing the
nutrient feed $\Pin$ simultaneously.
Which corresponds essentially to a multiobjective optimization
problem (MOP) defined as
\begin{equation}\label{eq:MOP}
\max\limits_{\uu\in\UU} \cb{\Pout(\uu), -\Pin(\uu)},
\end{equation}
where an admissible solution $\uu\in\UU$ that maximizes both
$\Pout$ and $-\Pin$ is sought.

\subsection{Bioreactor net profit}\label{subsec:net-profit}

We prove in sections \ref{sec:opti} and \ref{sec:feed} that there
is no feasible point that optimizes both criteria $\Pout$ and $\Pin$
simultaneously. This calls for a weaker notion of optimality that
is based on compromise and an external decision maker \cite{miettinen_1998}.

\begin{definition}[Pareto optimality]
A feasible solution is \textit{Pareto optimal} if none of the objective
functions could be improved without deteriorating at least one other objective.
This means that for the maximization of $\cb{f_1(\zz),\ldots,f_k(\zz)}$
functions for $\zz\in\ZZ$, a solution $\zz^*\in\ZZ$ is Pareto optimal if there
is no $\zz\in\ZZ$ such that $f_i(\zz^*)\leq f_i(\zz)$ for all $i\in\cb{1,\ldots,k}$ with
$f_j(\zz^*)<f_j(\zz)$ for some $j\in\cb{1,\ldots,k}$ \cite{miettinen_1998}.
\end{definition}
\begin{definition}[Pareto optimal front]
The Pareto optimal front (POF) is the set of Pareto optimal solutions.
\end{definition}

Solving the MOP consists of constructing the POF.
We achieve this by using the weighting method, which constructs
a scalar objective function that is a linear combination of
the objective functions, where the coefficients (or weights)
are positive and add up to $1$. In the case of two objectives,
the weights can be simply $\theta$ and $1-\theta$ with $\theta\in[0,1]$,
\cite{miettinen_1998}.
The weighting problem for \eqref{eq:MOP} is given as
\begin{equation}\label{eq:scalar-MOP}
\max\limits_{\uu\in\UU}\quad
P_{\theta}(\uu) = \theta\Pout(\uu) + (1-\theta)[-\Pin(\uu)].
\end{equation}
The weighting method allows obtaining \textit{locally} Pareto optimal
solutions. Nevertheless, stronger properties of the solutions
are discussed in sections \ref{sec:opti} and \ref{sec:feed}.

We note that the corresponding weighting problem \eqref{eq:scalar-MOP}
corresponds to maximizing the \textbf{net profit} of the bioreactor.
Let $\Win$ and $\Wout$ be the costs of a gram of glucose and a gram
of microalgae respectively. The net profit of the bioreactor is
the running cost difference between the harvested microalgae and
the bioreactor feed which is
$\Wout\Pout - \Win\Pin$.
The expression of $P_{\theta}$ can be found by dividing this expression
by $\Win + \Wout$ and setting $\theta=\Wout / (\Win + \Wout)$
which gives $\theta\in[0,1]$ for all positive $\Win$ and $\Wout$.

Apart from constructing the POF, the net profit provides insight
for the industrial biotechnology operators allowing to choose
the optimal values for the decision values for any given
set of costs $\Wout$ and $\Win$.

\section{Properties of the optimization problems for a fixed $\Sin$}\label{sec:opti}

In this section we restrict the problem to the control
$\uu=\transp{(\alpha,d)}\in\UU$, for a fixed $\Sin$,
where the set of admissible controls is defined as the set defining
presence of the coexistence equilibrium $\xx^*$, formalized as
\begin{equation}\label{eq:U-2D}
\UU = \cb{(\alpha,d)\in(0,1)\times\R_+^*,\hspace{0.5em}
\psi_\alpha^{-1}(d)<\Sin}.
\end{equation}

In \cite{asswad_cdc2024}, we have shown that the $\Pout(\alpha,d)$
is logarithmically bi-concave (\textit{i.e.} log-concave with respect
to $\alpha$ and to $d$ separately \cite{gorski_2007}) over $\UU$.
In~\cite{automatica_under_review} (an extended version
of~\cite{asswad_cdc2024}),
we have strengthened the result by showing that $\UU$ is a convex set,
and that $\Pout(\alpha,d)$ is logarithmically concave,
therefore admits a global maximum on $\UU$,
as shown in Figure \ref{fig:productivity}.
In contrast, the nutrient feed $\Pin(d)=d\cdot\Sin$ is an affine
function (for a fixed $\Sin$), and therefore has its optima
at the border of $\UU$.

\begin{figure}[t]
    \centering
    \includegraphics[width=\columnwidth]{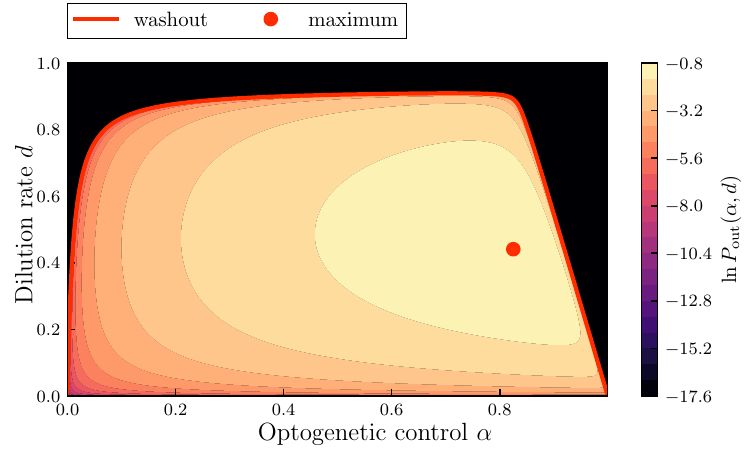}
    \vspace{-1em}
    \caption{Logarithm of the objective function $\Pout(\alpha,d)$ contours
    (borders of the domain of interest, ensuring nonzero algal biomass,
    indicated in solid red line), as well as its global maximum.} 
    \label{fig:productivity}
\end{figure}

\subsection{Bioreactor yield}

\begin{figure}[t]
    \centering
    \includegraphics[width=\columnwidth]{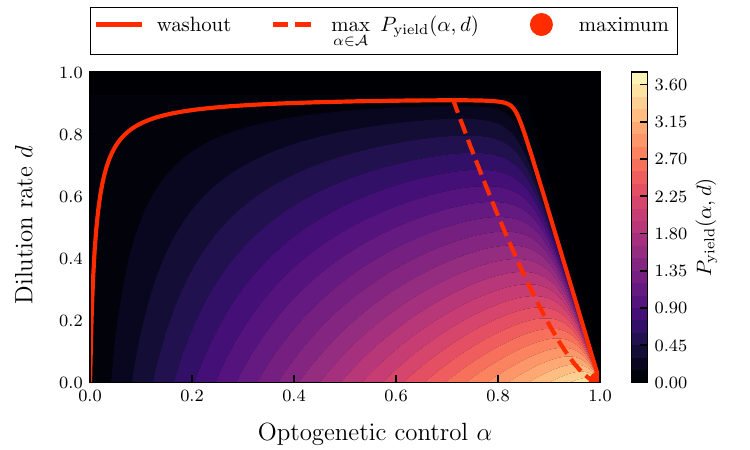}
    \vspace{-1em}
    \caption{Objective function $\Py(\alpha,d)$ contours, and
    $\max\limits_{\alpha\in\AA} \Py(\alpha,d)$ for all $d$ (dashed red line).} 
    \label{fig:yield}
\end{figure}

The bioreactor yield function $\Py(\alpha,d)$ does not enjoy
the same properties as the productivity $\Pout$. We can, nevertheless,
investigate its properties and exploit them to find its global
maximum analytically.

The expression of the yield is given as
\begin{equation}\label{Py}
\Py(\alpha,d) =
\frac{\alpha\beta\gamma\pp{\Sin - \psi_\alpha^{-1}(d)}}{\Sin\mu^{-1}(d)}.
\end{equation}
Let $\AA=(0,1)$ be an open interval, for a fixed
$\alpha\in\AA$ we define the open interval
$\DD_\alpha=(0,\psi_\alpha^{-1}(\Sin))$.

\begin{proposition}
For each fixed $\alpha\in\AA$, the function $\Py(\alpha,d)$
is strictly decreasing for all $d\in\DD_\alpha$.
\end{proposition}
\begin{proof}
We first notice that the affine function
${1/\mu^{-1}(d)=(\mumax-d)/\qmin\mumax}$
is positive and strictly decreasing on $\DD_\alpha$.
Furthermore, $\psi_\alpha^{-1}(d)$ is strictly increasing for
all $d\in\DD_\alpha$ because it is the positive linear combination
of strictly increasing functions composed with increasing functions in $d$,
therefore $d\mapsto\Sin - \psi_\alpha^{-1}(d)$
is positive and strictly decreasing on $\DD_\alpha$.
Finally, $d\mapsto \Py(\alpha,d)$ is the product of two positive strictly
decreasing functions therefore it is also strictly decreasing for all
$d\in\DD_\alpha$.
\end{proof}
It follows that for each $\alpha\in\AA$, the bioreactor yield
increases as $d$ decreases.
Let $\overline{\DD_\alpha}=[0,\psi_\alpha^{-1}(\Sin)]$ be the closure
of $\DD_\alpha$, it follows that for all $\alpha\in\overline\AA=[0,1]$,
\begin{equation}
\max_{d\in\overline{\DD_\alpha}} \Py(\alpha,d) = \Py(\alpha,0).
\end{equation}
Consequently, the maximization problem of $\Py(\alpha,d)$ over $\overline{\UU}$
reduces to the maximization $\Py(\alpha,0)$ over $\overline{\AA}$,
since
\begin{equation}
\max\limits_{(\alpha,d)\in\overline{\DD_\alpha}\times\overline{\AA}}
\Py(\alpha,d) = \max\limits_{\alpha\in\overline{\AA}}
\max\limits_{d\in\overline{\DD_\alpha}} \Py(\alpha,d).
\end{equation}
From \eqref{Py} and \eqref{eq:psi-alpha} we have
\begin{equation}\label{Py-d0}
\Py(\alpha,0) = \frac{\alpha\beta\gamma}{\qmin},
\end{equation}
which is affine and increasing with respect to $\alpha$ and reaches
its maximum at $\alpha=1$.
Finally, the bioreactor yield maximization problem over $\overline{\UU}$
has a trivial solution at $(1,0)$ as showin in Figure \ref{fig:yield}.

\begin{remark}\label{rk:yield}
For all $d\in\DD_\alpha$, the function $\Py(\alpha,d)$ is strictly concave
with respect to $\alpha$ over $\AA$. Consequently maximizing $\Py$
for a fixed $d\in\DD_\alpha$ gives a nontrivial solution
(Figure \ref{fig:yield}).
\end{remark}
\begin{proof}
This stems directly from the fact that $\Pout(\alpha,d)$ is strictly concave
with respect to $\alpha$ for any fixed $d$
(Lemma 2 in \cite{asswad_cdc2024})
since $\Py(\alpha,d)=\Pout(\alpha,d)/\Pin(d)$.
\end{proof}

In conclusion, maximizing the bioreactor yield degenerates to the special
case of a batch process ($d=0$). If continuous production is the goal, provided
a given extraction rate $d>0$ is of interest, Remark \ref{rk:yield} shows
that one can optimally and uniquely choose $\alpha$ for the given value
of $d$.

\subsection{Bioreactor net profit}

\begin{figure}[t]
    \centering
    \includegraphics[width=\columnwidth]{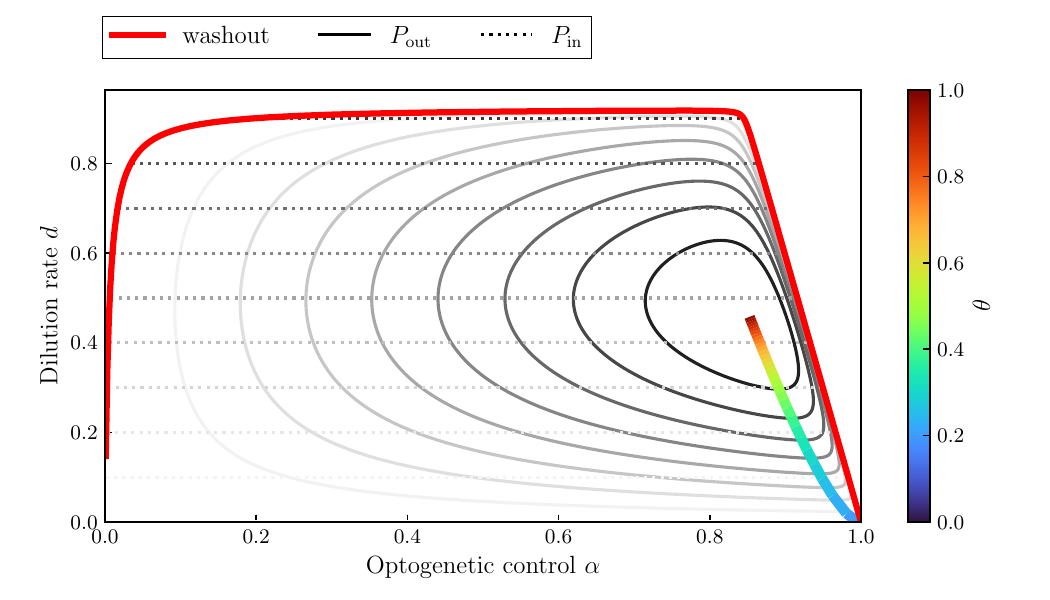}
    \caption{Contours of objective functions $\Pout$ and $\Pin$ over
    the set of admissible controls $\UU$.
    The color-gradient line is the set Pareto-optimal solutions for $\theta\in[0,1]$.}
    \label{fig:pareto-contours}
\end{figure}

\begin{figure}
    \centering
    \includegraphics[width=\columnwidth]{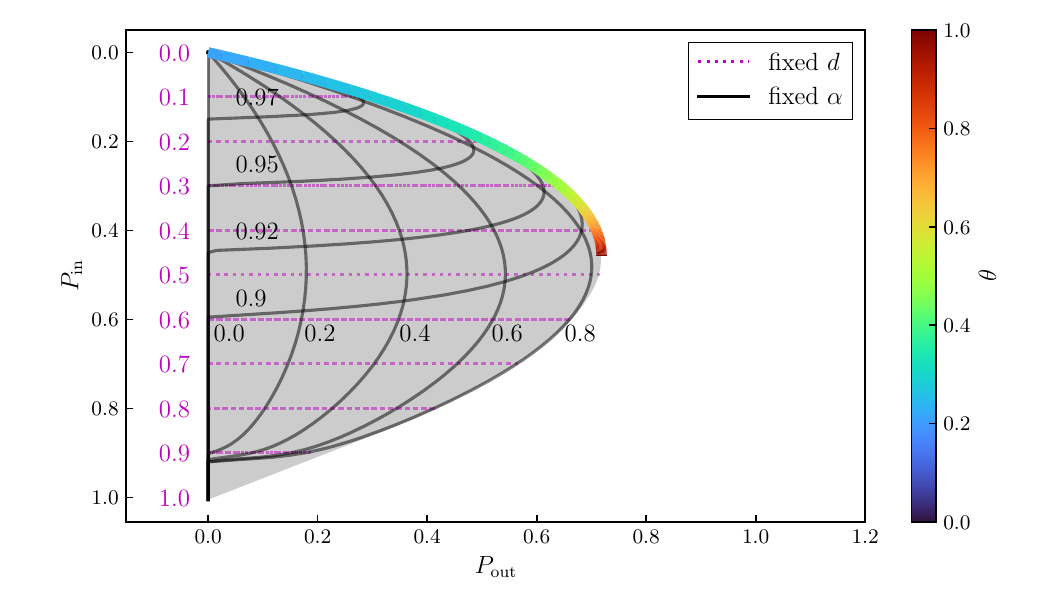}
    \caption{The image of $\UU$ in the multi-objective space; the color-gradient line
    represents the image of the POF through $P_\theta$ for $\theta\in[0,1]$,
    corresponding to the color-gradient line in Figure \ref{fig:pareto-contours}.}
    \label{fig:pof}
\end{figure}

As discussed in section \ref{sec:pb}, maximizing productivity
might not be the best course of action in practice, mainly for
economical reasons, and considering the running cost of the
bioreactor represents a matter of interest in biotechnology.
While the bioreactor yield takes the input feed into account,
it might not provide the needed insight in the context of
an industrial bioprocess, on top of the fact that we have shown
how its maximum does not correspond to our system of interest.
Hence, we consider instead the MOP defined in \eqref{eq:MOP}
and the associated weighted scalar problem \eqref{eq:scalar-MOP}
corresponding to the net profit.

The weighting method scalarizes the MOP in order to obtain \textit{locally}
Pareto optimal solutions. We now explore stronger properties of the solutions.

\begin{definition}
A decision vector $\zz^*\in\ZZ$ is \textbf{locally Pareto optimal} if there
exists $\delta>0$ such that $\zz^*$ is Pareto optimal in $\ZZ\cap B(\zz^*,\delta)$
where $B(\zz^*,\delta)$ denotes an open ball with a center $\zz^*$ and a radius
$\delta>0$ \cite{miettinen_1998}.
\end{definition}

\begin{lemma}\label{lem:global-POF}
Every \emph{locally} Pareto optimal solution is also \emph{globally} Pareto optimal
if the feasible region is convex and the objective functions are quasiconcave
with at least one strictly quasiconcave for the maximization problem
({Theorem~2.2.4}. in \cite{miettinen_1998}).
\end{lemma}

\begin{definition}[Quasiconcave function]
A function $f:\R^n\rightarrow\R$ is quasiconcave if its domain $\dom f$
and all its superlevel sets
$$C_\eta(f)=\cb{\zz\in\dom f, f(\zz)\geq \eta}$$
for $\eta\in\R$, are convex \cite{boyd_2004,avriel_2010}.
\end{definition}

\begin{proposition}\label{prop:global-POF}
Solutions of the weighted problem defined in \eqref{eq:scalar-MOP}
are globally Pareto optimal solutions for the feasible region $\UU$
defined in \eqref{eq:U-2D}.
\end{proposition}
\begin{proof}
Indeed, $\UU$ is a convex set, $\Pout(\alpha,d)$ is log-concave
therefore quasiconcave, and $\Pin(d)$ is an affine function
(non-constant) therefore strictly quasiconcave as well \cite{avriel_2010}.
Hence, the proof follows from Lemma \ref{lem:global-POF}.
\end{proof}

Proposition \ref{prop:global-POF} proves that a unique global optimum
exists for any given $\theta\in[0,1]$.
Figure \ref{fig:pareto-contours} illustrates the POF in the $(\alpha,d)$
plane and Figure \ref{fig:pof} shows the POF in the $(\Pout,\Pin)$ plane.
\section{Dependence of solutions on $\Sin$}\label{sec:feed}

The substrate feed $\Sin$ is an operational parameter
of the bioreactor that is commonly used as a control
input for bioprocesses. We have mentioned in section \ref{sec:pb}
that performance improves as it increases \cite{martinez_2023,mauri_2020}.
In this section we explore in detail the impact of adding
$\Sin$ as a third control variable on the optimization problem
properties and to the bioprocess.
We consider here the control $\ww=(\alpha,d,\Sin)\in\WW$ where
\begin{equation}\label{eq:U-3D}
\WW = \cb{(\alpha,d,\Sin)\in(0,1)\times(\R_+^*)^2,
\psi_\alpha^{-1}(d)<\Sin}
\end{equation}
is the set of admissible controls.

\subsection{Convexity properties}

Here we explore the convexity of the problems defined in
section \ref{sec:pb} with respect to the new control $\ww\in\WW$.

\begin{proposition}\label{nonconvex-3d}
The admissible control set $\WW$ is nonconvex.
\end{proposition}
\begin{proof}
Consider the function $g_0(\alpha,d)=\psi_\alpha^{-1}(d)$.
Its epigraph is clearly the closure of the set $\WW$
(i.e. the smallest closed set containing $\WW$, denoted $\overline{\WW}$).
We remind that a function is convex if and only if its epigraph is a convex
set \cite{boyd_2004}.
Nevertheless, for the given parameter set in Table \ref{tab:jc-params},
$g_0$ is not a convex function since the second order convexity condition
does not hold for all $(\alpha,d)\in\UU$: As demonstrated in Figure
\ref{fig:g0-hessian}, the Hessian of $g_0$ is 
non-definite over portions of the domain.

Therefore, $\epi(g_0) = \overline{\WW}$ is not a convex set.
\end{proof}
\begin{figure}[h]
    \centering
    \includegraphics[width=\columnwidth]{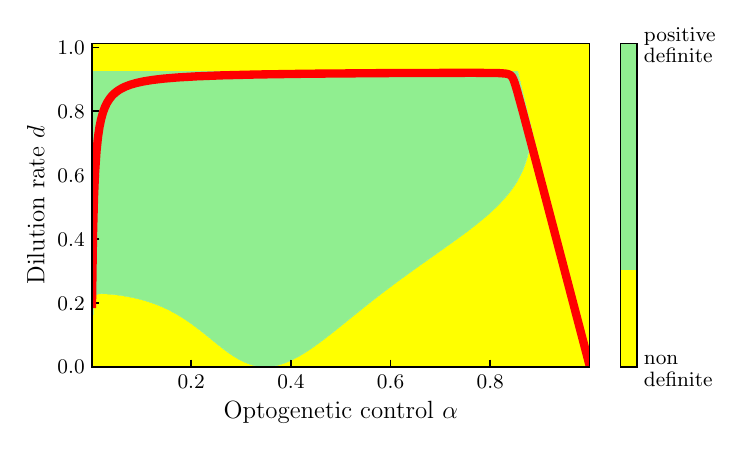}
    \caption{The positive-definiteness of the Hessian of the function
    $g_0(\alpha,d)$. Yellow regions indicate where the Hessian is non-definite, \text{i.e.} the Hessian is found having a positive and a negative eigenvalue.}
    \label{fig:g0-hessian}
\end{figure}
Since $\WW$ is not necessarily a convex set, the optimization problems
over $\WW$ are not convex problems regardless of the considered
objective function. We can however explore weaker properties.

\subsection{Problem biconvexity}

The set $\WW=\UU(z)\times\R_+^*$ is biconvex with
$\UU(z)=\cb{(\alpha,d)\in(0,1)\times\R_+^*,\psi_\alpha^{-1}(d)< z}$
for any $z>0$ \cite{gorski_2007}. This follows directly
from the convexity of $\UU(z)$ that corresponds to the admissible
controls set defined in section \ref{sec:opti} and the trivial
convexity of $\R_+^*$.

We now explore the biconvexity/biconcavity of the objective functions,
which would ensure the existence of global maxima. Uniqueness
however would not be guaranteed. Having comprehensively investigated
in Section \ref{sec:opti} the optima of the objective functions
$\Pout,\Pin,\Py,$ and $P_\theta$ with respect to $\uu=(\alpha,d)$
over $\UU$ for a fixed $\Sin>0$,
we now study their convexity/concavity with respect to $\Sin$ for
a fixed $(\alpha,d)$.

The functions $\Pout,\Pin,$ and $P_\theta$ are affine in $\Sin$,
therefore they do not have an optimum with respect to $\Sin$ over
an unbounded interval, therefore we introduce the open bounded set
$\BB(z)=\UU(z)\times (0,z)$ and its closure
$\overline{\BB(z)} = \overline{\UU(z)}\times[0,z]$ with
$\overline{\UU(z)} = \cb{(\alpha,d)\in[0,1]\times\R_+,\psi_\alpha^{-1}(d)\leq z}$.

Both $\Pout$ and $\Pin$ are affine and increasing with respect to $\Sin$,
therefore their optima lie on the border of $\overline{\BB(z)}$.
For a fixed pair $(\alpha,d)\in\overline{\UU(z)}$, $\Pout(\alpha,d,\Sin)$
reaches its maximum for $\Sin=z$ while $\Pin(d,\Sin)$ reaches its minimum
for $\Sin=0$.
We skip the optimization of the process yield since we have proven in
section \ref{sec:opti} that it has a trivial maximum that lies
on the border of $\UU(z)$ which corresponds to the singular case
of a batch ($d=0$) process.

\begin{figure}[t]
    \centering
    \includegraphics[width=\columnwidth]{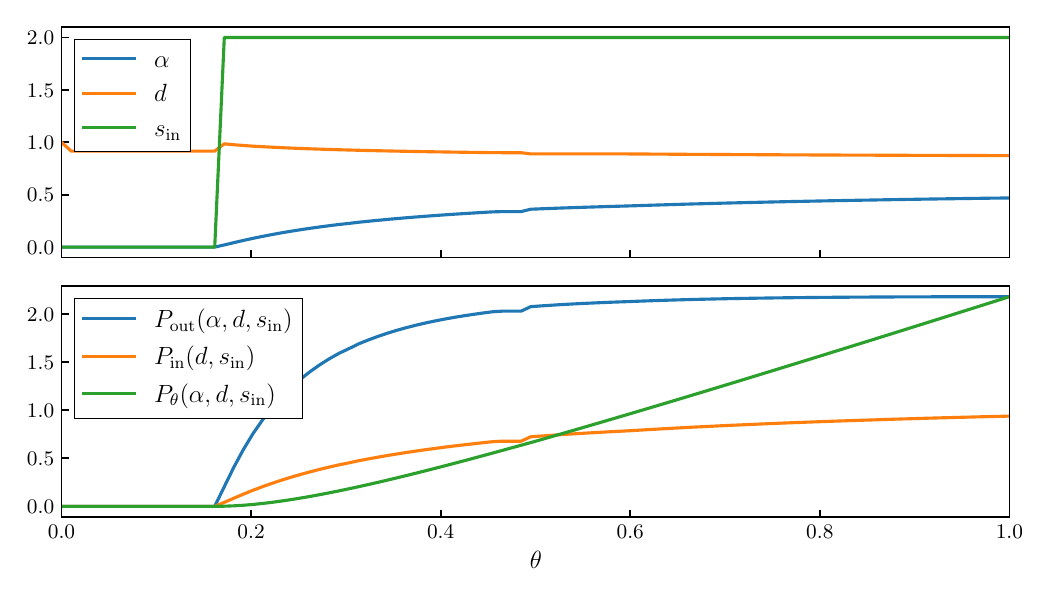}
    \caption{The top plot shows the optimal controls that corresponds to all
    $\theta\in[0,1]$ values. The bottom plot illustrates the corresponding objective
    values for the optimal controls in the top figure.}
    \label{fig:pof-theta}
\end{figure}

\subsection{Bioreactor net profit with respect to $\Sin$}

Considering the biconvex optimization problem over $\overline{\BB(z)}$,
the scalar weighting function of the MOP defined in \ref{sec:pb}
has a unique solution for a fixed $\Sin$ as established in \ref{sec:opti}.
Here we explore the effect of $\Sin$ on the bioreactor net profit
\begin{align}
P_{\theta}(\alpha,d,\Sin)
&= \theta\Pout(\alpha,d,\Sin) - (1-\theta)\Pin(d,\Sin)\\
&= \theta\frac{\alpha\beta\gamma d\pp{\Sin - \psi_\alpha^{-1}(d)}}{\mu^{-1}(d)} - (1-\theta)d\Sin \label{eq:P-theta-expr}
\end{align}
for a fixed $(\alpha,d)\in\overline{\UU(z)}$ and a given $\theta\in[0,1]$.
Since $P_\theta$ is affine with respect to $\Sin$, it reaches
therefore its maximum with respect to $\Sin$ on the border of the
closed interval $[0,z]$.
The function is either increasing or decreasing, depending on the
sign of the partial derivative of $P_\theta$ with respect to $\Sin$
(i.e. the coefficient of $\Sin$).
From \eqref{eq:P-theta-expr},
\begin{align}
\partials{P_\theta}{\Sin} 
&= \bb{\theta \pp{1 + \frac{\alpha\beta\gamma}{\mu^{-1}(d)}}
    - 1} d\nonumber\\
&= \pp{\theta - \theta_0(\alpha,d)}
    \pp{1 + \frac{\alpha\beta\gamma}{\mu^{-1}(d)}}d\label{eq:theta0}\\
&\text{with}~\theta_0(\alpha,d)
=\frac{1}{1 + \alpha\beta\gamma/\mu^{-1}(d)}.\nonumber
\end{align}
It follows that for a fixed $(\alpha,d)\in\UU$, $P_\theta$
is strictly increasing with respect to $\Sin$
if $\theta>\theta_0(\alpha,d)$ and decreasing otherwise.

\begin{figure}[t]
    \centering
    \includegraphics[width=\columnwidth]{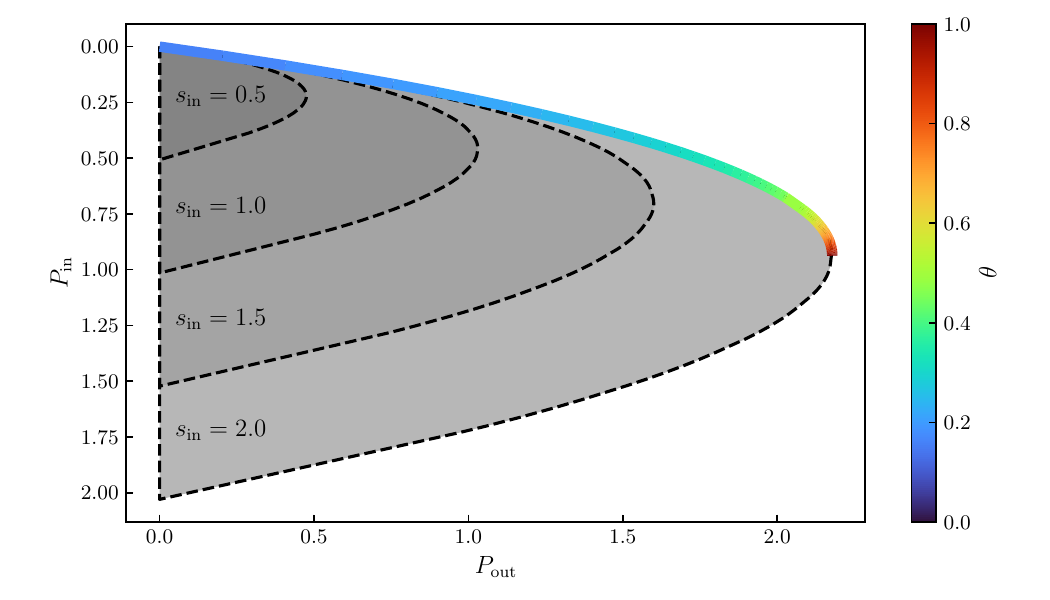}
    \caption{The image of $\overline{\BB(z)}$ with $z=2$ in
    the multi-objective space; the color-gradient line is the POF for $\theta\in[0,1]$.
    The shaded areas correspond to the set of reachable $\Pout$ and $\Pin$
    for different values of $\Sin$ (reachable sets for larger $\Sin$ including those for smaller $\Sin$).}
    \label{fig:pof-3d}
\end{figure}

\begin{proposition}\label{prop:theta0}
For all $(\alpha,d)\in\UU$, $0<\theta_0(\alpha,d)<1$.
\end{proposition}
\begin{proof}
Consider the functions $g_1$ and $g_2$ defined as
$$g_1(t) = \frac{1}{1 + t}~\text{and}~
g_2(\alpha,d)=\frac{\alpha\beta\gamma}{\mu^{-1}(d)}.$$
Then $\theta_0 = g_1\circ g_2$.
The function $g_2$ is strictly positive on $\UU$.
In addition, $g_1(\R_+^*)=(0,1)$.
Therefore, $\theta_0(\alpha,d)=g_1(g_2(\alpha,d))\in(0,1)$.
\end{proof}

It follows from \eqref{eq:theta0} and Proposition \ref{prop:theta0}
that, for a fixed $(\alpha,d)\in\UU$, $\partial P_\theta/\partial\Sin$
changes signs once at $\theta=\theta_0(\alpha,d)$.
Indeed, for values of $\theta<\theta_0$, $P_\theta(\alpha,d,\Sin)$ decreases
as $\Sin$ increases therefore the maximum of $P_\theta$ is reached at $\Sin=0$.
Otherwise, the net profit $P_\theta$ increases with $\Sin$ linearly
until it reaches its maximum at the border for $\Sin=z$.
Figure~\ref{fig:pof-theta} illustrates this result with respect to
all values of $\theta\in[0,1]$, for an upperbound on $s_\textit{in}$ fixed to $z=2$.
Clearly, $\Pout(\alpha,d,0)=\Pin(d,0)=0$
because $\UU$ is the empty set for $\Sin = 0$.
Otherwise, $\Sin$ takes its maximal value in the bounded
interval to maximize $P_\theta$.

Figure \ref{fig:pof-3d} shows the image of $\overline{\BB(z)}$
in the $(\Pout,\Pin)$ plane. It demonstrates the reachable
values of $\Pout$ and $\Pin$ for different fixed values of $\Sin$,
as well as the set of Pareto optimal solutions.

In conclusion, the added degree of freedom does not give optimal
solutions of the problem on the open set $\WW$.
Instead, it should be chosen a priori at its optimal boundary value
within the limits of feasibility and practicality to reach a target performance.
As a result of its optimal choice saturating to boundaries, for fixed values of
$\alpha$ and $d$, a discontinuity is introduced in its optimal choice as a function
of $\theta$.

\section{Conclusions}\label{sec:conclusion}

Starting from our algal-bacterial consortium model from \cite{asswad_cdc2024}
and the coexistence conditions established therein, we formulated optimization problems
for the system at its coexistence steady state that are relevant to biotechnological
applications.
These criteria were then integrated into a multi-objective optimization problem.
We demonstrated the uniqueness of solutions for these problems when controlling the
dilution rate $d$ and the optogenetic control $\alpha$, holding regardless of
specific biological parameter values,
and validated our findings through numerical simulations for realistic biological
parameter values.

Expanding the control framework, we introduced the bioreactor nutrient feed $\Sin$ as
a third control variable and analyzed its influence on the optimization problems
and their solutions.
In line with previous studies \cite{martinez_2023,mauri_2020}, for this class
of optimization criteria, our results indicated that enriching the nutrient feed
enhances system performance overall. The maximum allowable $\Sin$ should then be
chosen based on different biotechnological constraints (\textit{e.g.} maximal
allowable culture density, etc.).
Next, we defined the bioreactor net profit metric as an alternative to the
bioreactor yield, offering a weighted multi-objective perspective that is relevant
for industrial biotechnological applications. 
In the case of net profit, we showed that the role of the nutrient feed ($\Sin$)
is more complex (Proposition~\ref{prop:theta0} and subsequent discussion).
While this economy metric can be extended to other bioreactor models,
the uniqueness of solutions would need to be reexamined in each specific case.

Numerical simulations indicate that optimal steady-state solutions generally lie near washout boundaries. For future work, we will address the challenge of designing feedback control strategies that drive the system toward its optimal coexistence equilibrium and guarantee steady-state coexistence and optimality robustly, in spite of modeling uncertainties, limited state observations and inevitable sources of noise.

\section*{Acknowledgements}
This work was supported in part by the French national research agency (ANR) via project Ctrl-AB [ANR-20-CE45-0014].

\bibliography{ref.bib}
\bibliographystyle{IEEEtran}

\end{document}